\documentclass[reqno,12pt]{amsart}
\usepackage{amsfonts}
\usepackage{bbm}
\usepackage{} 
\numberwithin{equation}{section}
\usepackage{indentfirst}
 \usepackage{color}
\usepackage{amssymb}
\usepackage{mathrsfs}
\usepackage{xy}
\xyoption{all}
\def\Ext{\mbox{\rm Ext}\,} \def\Hom{\mbox{\rm Hom}} \def\dim{\mbox{\rm dim}\,} 
\def\lr#1{\langle #1\rangle}    \def\mod{\mbox{\rm \textbf{mod}}\,}\def\top{\mbox{\rm top}\,}
\def\Ker{\mbox{\rm Ker}\,}   \def\im{\mbox{\rm Im}\,} \def\Coker{\mbox{\rm Coker}\,}
\def\End{\mbox{\rm End}\,}\def\tw{\mbox{\rm tw}\,}
\def\Deg{\mbox{\rm \textbf{deg}}\,}
\def\ad{\mbox{\rm ad}\,}
\def\Aut{\mbox{\rm Aut}\,}\def\Dim{\mbox{\rm \textbf{dim}}\,}\def\A{\mathcal{A}\,} \def\H{\mathcal{H}\,}
\theoremstyle{plain} 
\newtheorem{theorem}{\bf Theorem}[section]
\newtheorem{lemma}[theorem]{\bf Lemma}
\newtheorem{corollary}[theorem]{\bf Corollary}
\newtheorem{proposition}[theorem]{\bf Proposition}

\theoremstyle{definition} 
\newtheorem{definition}[theorem]{\bf Definition}
\newtheorem{remark}[theorem]{\bf Remark}
\newtheorem{example}[theorem]{\bf Example}

\newcommand{\bt}{\begin{theorem}}
\newcommand{\et}{\end{theorem}}
\newcommand{\bl}{\begin{lemma}}
\newcommand{\el}{\end{lemma}}
\newcommand{\bd}{\begin{definition}}
\newcommand{\ed}{\end{definition}}
\newcommand{\bc}{\begin{corollary}}
\newcommand{\ec}{\end{corollary}}
\newcommand{\bp}{\begin{proof}}
\newcommand{\ep}{\end{proof}}
\newcommand{\bx}{\begin{example}}
\newcommand{\ex}{\end{example}}
\newcommand{\br}{\begin{remark}}
\newcommand{\er}{\end{remark}}
\newcommand{\be}{\begin{equation}}
\newcommand{\ee}{\end{equation}}
\newcommand{\ba}{\begin{align}}
\newcommand{\ea}{\end{align}}
\newcommand{\bn}{\begin{enumerate}}
\newcommand{\en}{\end{enumerate}}
\newcommand{\bcs}{\begin{cases}}
\newcommand{\ecs}{\end{cases}}
\makeatletter
\renewcommand{\section}{\@startsection{section}{1}{0mm}
  {-\baselineskip}{0.5\baselineskip}{\bf\leftline}}
\makeatother

\begin{document}

\title[Minimal generators of Hall algebras]{Minimal generators of Hall algebras of\\ 1-cyclic perfect complexes} 

\author[Haicheng Zhang]{{Haicheng Zhang}} 


%
\address{
Institute of Mathematics, School of Mathematical Sciences, Nanjing Normal University,
Nanjing 210023, P. R. China.\endgraf
}
\email{zhanghc@njnu.edu.cn}


\begin{abstract}
Let $A$ be the path algebra of a Dynkin quiver over a finite field, and let $C_1(\mathscr{P})$ be the category of 1-cyclic complexes of projective $A$-modules. In the present paper, we give a PBW-basis and a minimal set of generators for the Hall algebra $\H(C_1(\mathscr{P}))$ of $C_1(\mathscr{P})$. Using this PBW-basis, we firstly prove the degenerate Hall algebra of $C_1(\mathscr{P})$ is the universal enveloping algebra of the Lie algebra spanned by all indecomposable objects. Secondly, we calculate the relations in the generators in $\H(C_1(\mathscr{P}))$, and obtain quantum Serre relations in a quotient of certain twisted version of $\H(C_1(\mathscr{P}))$. Moreover, we establish relations between the degenerate Hall algebra, twisted Hall algebra of $A$ and those of $C_1(\mathscr{P})$, respectively.
\end{abstract}

\maketitle

\section{Introduction}
Let $A$ be always a finite dimensional hereditary algebra over a finite field. In what follows, all modules are assumed to be finite dimensional. In 2011, Bridgeland \cite{Br} considered the Hall algebra of 2-cyclic complexes of projective $A$-modules. By taking some localization and reduction, he achieved an algebra called the Bridgeland Hall algebra of $A$. He proved that the quantum enveloping algebra associated to $A$ is embedded into the Bridgeland Hall algebra of $A$. This provides a beautiful realization of the full quantum enveloping algebra by Hall algebras. Bridgeland \cite{Br} made a statement without proofs that the Bridgeland Hall algebra of $A$ is isomorphic to the Drinfeld double of its extended Ringel--Hall algebra, which is later proved by Yanagida in \cite{Yan} (see also \cite{ZHC}). Inspired by Bridgeland's work, for each positive integer $m\geq2$, Chen and Deng \cite{ChenD} considered the Hall algebra of $m$-cyclic complexes of projective $A$-modules. For a representation-finite hereditary algebra $A$ they proved the existence of Hall polynomials in the category of $m$-cyclic complexes of projective $A$-modules; using the Hall polynomials, they defined the generic Bridgeland Hall algebra of $2$-cyclic complexes, and showed that it contains a subalgebra isomorphic to the integral form of the quantum enveloping algebra associated to $A$; in particular, for the degenerate case, this provides a realization of the simple Lie algebra associated to $A$. For $m>2$, the algebra structure of the Bridgeland Hall algebra $DH_m(A)$ of $m$-cyclic complexes of projective $A$-modules has a characterization in \cite{ZHC2}, in particular, it is proved that there exist Heisenberg double structures in $DH_m(A)$.

On the other hand, the most difficult case is the $1$-cyclic complex case. Let $C_1(\mathscr{P})$ be the category of $1$-cyclic complexes of projective $A$-modules, and for each $A$-module $M$ the corresponding $1$-cyclic complex of projective $A$-modules is denoted by $C_M$. Then for any two $A$-modules $M, N$, $$\Ext_{C_1(\mathscr{P})}^1(C_M,C_N)\cong \Hom_A(M,N)\oplus\Ext_A^1(M,N)~(\text{see~Lemma~\ref{kuozhang}}).$$ That is, the exact structure of $C_1(\mathscr{P})$ is more complicated than that of the $A$-module category. Actually, let $A=kQ$ be the path algebra of a finite acyclic quiver $Q$, and let $\Lambda=k[x]/[x^2]$ be the algebra of dual numbers, then the category of $1$-cyclic complexes of $A$-modules is equivalent to the category of modules over the path algebra $\Lambda Q$, and $C_1(\mathscr{P})$ is exactly the category of Gorenstein projective $\Lambda Q$-modules (cf. \cite{RZ}). Ringel and Zhang pointed out in \cite{RZ} that ``The class of 1-Gorenstein algebras is a class of algebras which includes both the hereditary and the self-injective algebras---two classes of algebras whose representations have been investigated very thoroughly and have been shown to be strongly related to Lie theory". Thus, one might hope that all the 1-Gorenstein algebras should have a close relation with Lie theory. As a representative of the class of 1-Gorenstein algebras, the path algebra $\Lambda Q$ is $1$-Gorenstein. Ringel and Zhang \cite{RZ} have showed that the Kac theorem yields a correspondence between the isoclasses (isomorphism classes) of indecomposable objects in the stable category of $C_1(\mathscr{P})$ and the positive roots of the Kac--Moody algebra associated to $A$.

Inspired by the above correspondence given by Ringel and Zhang, for a hereditary algebra $A$ of Dynkin type, some research on the Lie theory of $C_1(\mathscr{P})$ has been undertaken in \cite{RSZ}. They gave a minimal set of generators for the Lie algebra $\tilde{\mathfrak{n}}^+$ spanned by the isoclasses of indecomposable non-projective objects in $C_1(\mathscr{P})$, and calculated some relations in these generators.
In particular, for a bipartite quiver, the Lie algebra $\tilde{\mathfrak{n}}^+$ provides a realization of the positive part of the corresponding simple Lie algebra; for a linearly oriented quiver of type $\mathbb{A}$, it provides a realization of free $2$-step nilpotent Lie algebra. Moreover, for all the quivers of type $\mathbb{A}$, the generators and generating relations for $\tilde{\mathfrak{n}}^+$ have been determined. This achieves the desire that $C_1(\mathscr{P})$ should be related to Lie theory.

As mentioned above, compared with the $A$-module category, the exact structure of $C_1(\mathscr{P})$ is ``twisted" severely. In other words, the algebra structure of the Hall algebra of $C_1(\mathscr{P})$ becomes more complicated. In the present paper, we will give some characterizations on the Hall algebra $\H(C_1(\mathscr{P}))$ of $C_1(\mathscr{P})$. Explicitly, we give a PBW-basis for the Hall algebra of $C_1(\mathscr{P})$. Then for a hereditary algebra of Dynkin type, using this PBW-basis, we firstly prove the degenerate Hall algebra of $C_1(\mathscr{P})$ is the universal enveloping algebra of the Lie algebra spanned by all isoclasses of indecomposable objects; secondly, we give a minimal set of generators for $\H(C_1(\mathscr{P}))$, calculate the relations in these generators in $\H(C_1(\mathscr{P}))$, and obtain some fundamental relations in a quotient of $\H(C_1(\mathscr{P}))$. In order to get quantum Serre relations, we define a twisted version of $\H(C_1(\mathscr{P}))$. Moreover, we establish relations between the degenerate Hall algebra, twisted Hall algebra of $A$ and those of $C_1(\mathscr{P})$, respectively.

Let us fix some notations used throughout the paper. $k=F_q$ is always a finite field with
$q$ elements, $v=\sqrt{q}$, and $Q(v)$ is the rational function field of $v$. Let $Q$ be a finite acyclic quiver with $n$ vertices, and $A$ be the path algebra of $Q$ over $k$. Denote by $\mod A$ the category of finite dimensional (left) $A$-modules, and by $\mathscr{P}\subset\mod A$ the subcategory of projective $A$-modules. The bounded derived category and Grothendieck group of $\mod A$ are denoted by $D^b(A)$ and $K(A)$, respectively; and for any $M\in\mod A$ we denote by $\hat{M}$ the image of $M$ in $K(A)$. For each vertex $i$ of $Q$, we denote by $S_i$ the corresponding simple $A$-module, and by $P_i$ the projective cover of $S_i$. For an $A$-module $M$ and a positive integer $m$, $mM$ stands for the direct sum of $m$ copies of $M$. We denote by $[X]$ the isoclass of an object $X$ in an additive category. For a finite set $S$, we denote by $|S|$ its cardinality. For a Lie algebra $\mathfrak{g}$, we denote by $U(\mathfrak{g})$ its universal enveloping algebra (refer to \cite{Hum} for Lie theory).

\section{Preliminaries}
In this section, we collect some definitions and properties of $1$-cyclic complexes and Hall algebras.

\subsection{$1$-cyclic complexes}
A \emph{1-cyclic complex} of $A$-modules is by definition a pair $M^\cdot=(M,d)$ where $M$ is an $A$-module and $d$ is an endomorphism of $M$ satisfying $d^2=0$. Let $(M,d)$ and $(M',d')$ be two $1$-cyclic complexes of $A$-modules, a \emph{morphism} $f:(M,d)\rightarrow (M',d')$ is given by a homomorphism $f:M\rightarrow M'$ of $A$-modules such that $d'f=fd$. Two morphisms $f, g:(M,d)\rightarrow (M',d')$ are said to be \emph{homotopic} provided there exists a homomorphism $s:M\rightarrow M'$ of $A$-modules such that $f-g=sd+d's$. For each 1-cyclic complex $M^\cdot=(M,d)$ of $A$-modules, its homology $H_0(M^\cdot):=\Ker d/ \im d$. We denote by $C_1(\mod A)$ the category of $1$-cyclic complexes of $A$-modules. Let $C_1(\mathscr{P})$ be the subcategory of $C_1(\mod A)$ consisting of 1-cyclic complexes of projective $A$-modules, and denote by ${K}_1(\mathscr{P})$ the homotopy category obtained from $C_1(\mathscr{P})$ by identifying homotopic morphisms.
It is similar to the ordinary bounded complexes that we have a shift functor
$[1]:C_1(\mod A)\longrightarrow C_1(\mod A)$ defined by $M^\cdot[1]:=(M,-d)$, where $M^\cdot=(M,d)$. It is well to be reminded that $C_1(\mathscr{P})$ is a Frobenius exact category, whose stable category coincides with the homotopy category ${K}_1(\mathscr{P})$.

The following lemma is significant to the calculation of extension groups in $C_1(\mathscr{P})$ (cf. \cite{Gor2,ChenD,Zhao}).
\bl\label{Ext} If $X^\cdot, Y^\cdot\in C_1(\mathscr{P})$, then $$\Ext_{C_1(\mathscr{P})}^1(X^\cdot,Y^\cdot)\cong
\Hom_{K_1(\mathscr{P})}(X^\cdot, Y^\cdot[1]).$$
\el
Given a morphism
$f: \Omega\rightarrow P$ of projective $A$-modules, one defines a $1$-cyclic complex
$$C_f={\left(P\oplus \Omega,\begin{pmatrix}0&f\\0&0\end{pmatrix}\right)}\in C_1(\mathscr{P}).$$
Hence, for each projective $A$-module $P$, we have a $1$-cyclic complex $K_P:=C_{{\rm Id}_P}\in C_1(\mathscr{P})$. For each $A$-module $M$, we fix a minimal projective resolution of $M$:
\begin{equation}\label{mini proj res}
0 \longrightarrow \Omega_M \stackrel{\delta_M}{\longrightarrow} P_M \stackrel{\epsilon_M}{\longrightarrow} M \longrightarrow 0.
\end{equation} Then we set $C_M:=C_{\delta_M}$. Since the minimal projective resolution is unique up to isomorphism, $C_M$ is well-defined up to isomorphism.

The following lemma gives characterizations of the classification of indecomposable objects in $C_1(\mathscr{P})$ and the structure of its homotopy category.
\bl\rm{(\cite[Theorem 1]{RZ})}\label{indec. obj.s}

$(1)$ The objects $C_M$ and $K_P$, where $M$ is an indecomposable $A$-module, and $P$ is an indecomposable projective $A$-module, provide a complete set of indecomposable objects in $C_1(\mathscr{P})$. Moreover, all $K_P$ are exactly the whole indecomposable projective-injective objects in $C_1(\mathscr{P})$.

$(2)$ The homotopy category $K_1(\mathscr{P})$ is equivalent to the orbit category $D^b(A)/[1]$ as triangulated categories.
\el

Combining Lemma \ref{Ext} with Lemma \ref{indec. obj.s}, we obtain the following
\begin{lemma}\label{kuozhang}
For any $X^\cdot, Y^\cdot\in C_1(\mathscr{P})$,
$$\Ext^1_{C_1(\mathscr{P})}(X^\cdot,Y^\cdot)\cong\Hom_A(H_0(X^\cdot),H_0(Y^\cdot))\oplus\Ext^1_{A}(H_0(X^\cdot),H_0(Y^\cdot)).$$
\end{lemma}
\begin{proof}
By Lemma \ref{indec. obj.s}, we write $X^\cdot=C_M\oplus K_P$ and $Y^\cdot=C_N\oplus K_\Omega$ for some $M,N\in\mod A$, and $P,\Omega\in\mathscr{P}$. Then
\begin{equation*}\begin{split}
\Ext^1_{C_1(\mathscr{P})}(X^\cdot,Y^\cdot)&=\Ext^1_{C_1(\mathscr{P})}(C_M\oplus K_P,C_N\oplus K_\Omega)\\
&\cong\Ext^1_{C_1(\mathscr{P})}(C_M,C_N)\\
&\cong\Hom_{K_1(\mathscr{P})}(C_M,C_N[1])\\
&\cong\Hom_{D^b(A)/[1]}(M,N[1])\\
&\cong\bigoplus_{i\in\mathbb{Z}}\Hom_{D^b(A)}(M,N[i+1])\\
&\cong\Hom_A(M,N)\oplus\Ext^1_{A}(M,N)\\
&\cong\Hom_A(H_0(X^\cdot),H_0(Y^\cdot))\oplus\Ext^1_{A}(H_0(X^\cdot),H_0(Y^\cdot)).
\end{split}\end{equation*}
\end{proof}

The following well-known result is also needed in the sequel.
\begin{lemma}\label{end}
For each short exact sequence of $A$-modules
$$\xi: 0\longrightarrow M\longrightarrow L\longrightarrow N\longrightarrow0,$$
we have that $\dim_k\End_A L\leq\dim_k\End_A (M\oplus N)$, and ``=" holds if and only if $\xi$ is splitting.
\end{lemma}

In order to have an intuitive cognition of the structure of $C_1(\mathscr{P})$, we give an example of the Auslander--Reiten quiver of $C_1(\mathscr{P})$. We recommend \cite{RZ} for the Auslander--Reiten theory of $C_1(\mathscr{P})$.

\begin{example}\label{lizi}
Let $Q$ be the quiver of type $\mathbb{A}_3$
$$1\longrightarrow 2\longrightarrow 3.$$
The Auslander--Reiten quiver of $C_1(\mathscr{P})$ is as follows:
\be\label{AR-quiver}\xymatrix@!=0.8pc{\ar@{.}[d]&K_{P_1}\ar[rd]&&&\\
C_{S_1}\ar@{.}[dd]\ar[ur]\ar[rd]\ar@{--}[rr]&&C_{P_1}\ar[rd]\ar@{--}[rr]&&C_{P_3}\ar@{.>}[u]\ar@{.}[dd]\\
&C_{P_2}\ar[ur]\ar[rd]\ar@{--}[rr]&&C_{I_2}\ar[ur]\ar[rd]&\\
C_{P_3}\ar@{.>}[d]\ar[ru]\ar[rd]\ar@{--}[rr]&&C_{S_2}\ar[ur]\ar[rd]\ar@{--}[rr]&&C_{S_1}\ar@{.}[d]\\
&K_{P_3}\ar[ur]&&K_{P_2}\ar[ur]&}\ee where the horizontal dashed lines denote different $\tau$-orbits, and we should join the vertical dotted lines such that their directions match, that is, the Auslander--Reiten quiver of $C_1(\mathscr{P})$ is like a Mobius strip. All $K_{P_i}$ are projective and injective in $C_1(\mathscr{P})$, and thus their orbits are themselves.
\end{example}

\subsection{Hall algebras}
Let $\mathcal{A}$ be a finitary and skeletally small exact $k$-category and let $W_{XY}^Z$ denote the set $\{(\varphi,\psi)~|~0\rightarrow Y\xrightarrow{\varphi} Z\xrightarrow{\psi} X\rightarrow 0~\text{is~exact~in}~\mathcal{A}\}$. The group $G:=\Aut{X}\times\Aut{Y}$ acts on $W_{XY}^Z$ via
$$\xymatrix{&0\ar[r]&Y\ar[r]^{\varphi}\ar[d]_{g}&Z\ar[r]^{\psi}\ar@{=}[d]&X\ar[r]\ar[d]^{f}&0\\
&0\ar[r]&Y\ar[r]^{\overline{\varphi}}&Z\ar[r]^{\overline{\psi}}&X\ar[r]&0.}$$
That is, for any $(\varphi,\psi)\in W_{XY}^Z$ and $(f,g)\in G$, $(f,g)\cdot (\varphi,\psi)=(\varphi g^{-1},f\psi)$.
We denote the set of $G$-orbits by $V_{XY}^Z$. Since $\varphi$ is monic and $\psi$ epic this action is free, and we define $$F_{X Y}^{Z}:=|V_{X Y}^{Z}|=\frac{|W_{X Y}^{Z}|}{|\Aut X|\cdot|\Aut Y|}.$$

By the Riedtmann--Peng formula \cite{Riedt,Peng}, we know that for any objects $X,Y,Z$ in $\A$
$$F_{X Y}^{Z}=\frac{|\Ext_{\A}^1(X,Y)_{Z}|}{|\Hom_{\A}(X,Y)|}\cdot
\frac{|\Aut{Z}|}{|\Aut{X}|\cdot|\Aut{Y}|},$$ where $\Ext_{\A}^1(X,Y)_{Z}$ denotes the subset of $\Ext_{\A}^1(X,Y)$ consisting of equivalence classes of exact sequences of the form $0\rightarrow Y\rightarrow Z\rightarrow X\rightarrow 0.$

\begin{definition}\label{def of Hall}
Let $\A$ be a finitary and skeletally small exact $k$-category.
The \emph{Hall algebra} $\H(\A)$ of $\A$ is the vector space over $\mathbb{C}$ with basis the isoclasses $[X]$ of objects in $\A$, and with multiplication defined by
\[[X]\cdot[Y] = \sum\limits_{[Z]} F_{X,Y}^{Z} [Z].\]
\end{definition}
In particular, if $\A=\mod A$, we obtain the Hall algebra $\H(\mod A)$ of $A$, which is also denoted by $\H(A)$; if $\A=C_1(\mathscr{P})$, we obtain the Hall algebra $\H(C_1(\mathscr{P}))$ of $C_1(\mathscr{P})$.

\subsection{Degenerate Hall algebras and Lie algebras associated to $C_1(\mathscr{P})$}
In this subsection, let $Q$ be a Dynkin quiver, that is, $Q$ is of type $\mathbb{A}\mathbb{D}\mathbb{E}$, and let $\Gamma$ be the underlying graph of $Q$. For each prime power $q$ ($\neq 1$ by convention), we  denote by  $A=A(q)$ the path algebra of $Q$ over $k=\mathbb{F}_q$.

By the well-known theorem of Gabriel \cite{Gabriel1,Gabriel2}, the correspondence $M\mapsto\Dim M$ induces a bijection between the set of isoclasses of indecomposable $A$-modules and the set of positive roots $\Phi^{+}$ of the simple Lie algebra $\mathfrak{g}$ associated with $\Gamma$. For each $\alpha\in\Phi^{+}$, let $M_q(\alpha)$ denote a representative of the corresponding indecomposable $A$-modules. For each $1\leq i\leq n$, let $\beta_i$ be the root in $\Phi^{+}$ such that $M_q(\beta_i)\cong P_i$.

By Lemma \ref{indec. obj.s}, the set $$\{C_{M_q(\alpha)}, K_{M_q(\beta_i)}~|~\alpha\in\Phi^{+}, 1\leq i\leq n\}$$ is a complete set of indecomposable objects in $C_1(\mathscr{P})$. Set $I=\{1,\cdots,n\}$, $\mathcal {I}_1(\Gamma)=\Phi^{+}\cup I$, and define $\mathfrak{P}_1(\Gamma)=\{\lambda:\mathcal {I}_1(\Gamma)\rightarrow\mathbb{N}\}$. By the Krull--Schmidt theorem, the correspondence sending $\lambda\in\mathfrak{P}_1(\Gamma)$ to $$C(\lambda)=C_q(\lambda)=(\bigoplus\limits_{\alpha\in\Phi^{+}}\lambda(\alpha){C_{M_q(\alpha)}})\bigoplus(\bigoplus\limits_{1\leq i\leq n}\lambda(i){K_{M_q(\beta_i)}})$$ induces a bijection from $\mathfrak{P}_1(\Gamma)$ to the set of isoclasses of objects in $C_1(\mathscr{P})$. An element $\lambda\in\mathfrak{P}_1(\Gamma)$ is said to be \emph{indecomposable} if $C_q(\lambda)$ is indecomposable and \emph{decomposable} otherwise.
\begin{remark}
There is a bijection from the set of functions $\lambda: \Phi^+\rightarrow\mathbb{N}$ to the set of isoclasses of $A$-modules by sending $\lambda\mapsto[M_q(\lambda)]$, where
$$M_q(\lambda)=\bigoplus\limits_{\alpha\in\Phi^+}\lambda(\alpha)M_q(\alpha)\in\mod A.$$
\end{remark}

\begin{proposition}{\rm(\cite[Theorem 3.6]{RSZ})}\label{poly}
For any $\lambda, \mu, \nu\in\mathfrak{P}_1(\Gamma)$, there exists a polynomial $\psi_{\mu \nu}^{\lambda}(x)\in\mathbb{Z}[x]$ such that for each prime power $q$,
$$\psi_{\mu \nu}^{\lambda}(q)=F_{C_q(\mu) C_q(\nu)}^{C_q(\lambda)}.$$
\end{proposition}
The polynomials $\psi_{\mu \nu}^{\lambda}$ in Proposition \ref{poly} are called \emph{Hall polynomials}.

The \emph{degenerate Hall algebra} $\H_1(C_1(\mathscr{P}))$ of $C_1(\mathscr{P})$ is the same as $\H(C_1(\mathscr{P}))$ as vector spaces, but with multiplication defined by $$[C_q(\mu)][C_q(\nu)]= \sum\limits_{\lambda\in\mathfrak{P}_1(\Gamma)} \psi_{\mu \nu}^{\lambda}(1) [C_q(\lambda)].$$
\begin{remark}
By the well-known result of Ringel \cite{R91a}, the Hall polynomials exist in $\mod A$. Hence, we can similarly define the degenerate Hall algebra $\H_1(A)$ of $A$.
\end{remark}

Let $\tilde{\mathfrak{n}}$ be the vector space over $\mathbb{C}$ spanned by the isoclasses of indecomposable objects in $C_1(\mathscr{P})$, and let $\tilde{\mathfrak{n}}^{+}$ (resp. $\mathfrak{h}$) be the subspace of $\tilde{\mathfrak{n}}$ spanned by the isoclasses of indecomposable and non-projective (resp. projective) objects in $C_1(\mathscr{P})$.

\begin{proposition}{\rm(\cite[Proposition 4.1]{RSZ})}
$\tilde{\mathfrak{n}}$ is a Lie algebra with Lie bracket defined by
$$[[C_q(\mu)], [C_q(\nu)]]=\sum\limits_{\lambda\in\mathfrak{P}_1(\Gamma)}(\psi_{\mu \nu}^{\lambda}(1)-\psi_{\nu \mu}^{\lambda}(1))[C_q(\lambda)]$$ for any indecomposable $\mu, \nu\in\mathfrak{P}_1(\Gamma)$, and $\tilde{\mathfrak{n}}^{+},\mathfrak{h}$ are two Lie subalgebras (ideals) of~$\tilde{\mathfrak{n}}$.
\end{proposition}
We remark that as Lie algebras $\tilde{\mathfrak{n}}\cong\tilde{\mathfrak{n}}^{+}\oplus\mathfrak{h}$, and $[\tilde{\mathfrak{n}},\mathfrak{h}]=0$. Thus by the PBW-basis theorem of universal enveloping algebras, we have the decomposition of the universal enveloping algebra $U(\tilde{\mathfrak{n}})\cong U(\tilde{\mathfrak{n}}^{+})\otimes U(\mathfrak{h})$, and $U(\mathfrak{h})\cong k[x_1,x_2,\cdots,x_n]$. Actually, since $[\tilde{\mathfrak{n}}^{+},\mathfrak{h}]=0$, it is easy to see that \begin{equation}\label{fi}
g: U(\tilde{\mathfrak{n}}^{+})\otimes k[x_1,x_2,\cdots,x_n]\longrightarrow U(\tilde{\mathfrak{n}}),~~[C_M]\mapsto [C_M],x_i\mapsto K_{P_i},\end{equation}
where $M\in\mod A$ is indecomposable and $1\leq i\leq n$, is an isomorphism of algebras.

In order to state the following theorem, which is the main result of \cite{RSZ}, we introduce the \emph{path matrix} $E=(a_{ij})_{n\times n}$ of the Dynkin quiver $Q$: if there is a path between $i$ and $j$ in $Q$, say from $i$ to $j$, then $a_{ij}=1$ and $a_{ji}=-1$; otherwise, $a_{ij}=a_{ji}=0$.
\begin{theorem}{\rm(\cite[Theorem 4.3]{RSZ})}\label{cpg}
The Lie algebra $\tilde{\mathfrak{n}}^{+}$ is generated by $\{[C_{P_i}]~|~1\leq i\leq n\}$, and these generators satisfy the following relations:
\begin{itemize}
\item[(a)] If $|a_{ij}|=1$, $(\ad \epsilon_i)^2(\epsilon_j)=(\ad \epsilon_j)^2(\epsilon_i)=0$;
\item[(b)] If $a_{ij}a_{jk}=1$, $[\epsilon_i,[\epsilon_j,\epsilon_k]]=[\epsilon_k,[\epsilon_i,\epsilon_j]]=0$;
\item[(c)] If $a_{ij}=0$, $[\epsilon_i,\epsilon_j]=0$,
\end{itemize}
where $\epsilon_i:=[C_{P_i}]$ for each $1\leq i\leq n$.
\end{theorem}

\section{PBW-basis of $\H(C_1(\mathscr{P}))$}
In this section, we will use the intrinsic filtered structure of $\H(C_1(\mathscr{P}))$ to prove its PBW-basis theorem. The PBW-basis theorem of the Hall algebra of an algebra is proved in \cite{GuoP}, and it is generalized to the Hall algebra of a finitary exact category in \cite{BG}.

Let us state the PBW-basis theorem of $\H(C_1(\mathscr{P}))$ as follows:
\begin{theorem}\label{PBW}
$S:=\{[X^\cdot]~|~X^\cdot\in C_1(\mathscr{P})~\text{is~indecomposable}\}$ is a universal PBW-basis of the Hall algebra $\H(C_1(\mathscr{P}))$. That is, for any total order $\leq$ on $S$, the monomials $[X^\cdot_1]^{a_1}[X^\cdot_2]^{a_2}\cdots[X^\cdot_N]^{a_N}$, where $[X^\cdot_1]\leq[X^\cdot_2]\leq\cdots\leq[X^\cdot_N]$ are the whole isoclasses of pairwise non-isomorphic indecomposable objects and all $a_i\in\mathbb{N}$, together with the unit $[0]$, form a basis for $\H(C_1(\mathscr{P}))$.
\end{theorem}
\begin{proof}
For any short exact sequence in $C_1(\mathscr{P})$
$$0\longrightarrow Y^\cdot\longrightarrow Z^\cdot\longrightarrow X^\cdot\longrightarrow 0,$$ it induces a long exact sequence in homology
\begin{equation}\label{czh}\xymatrix{0\ar[r]&K\ar[r]&H_0(Y^\cdot)\ar[r]^-{\varphi}&H_0(Z^\cdot)\ar[r]^-{\psi}&H_0(X^\cdot)\ar[lld]_-{f}\ar[r]&Q\ar[r]&0\\
&&H_0(Y^\cdot)\ar[r]^-{\varphi}&H_0(Z^\cdot)\ar[r]^-{\psi}&H_0(X^\cdot).&&}\end{equation}
Since $\Ker \varphi\cong\im f\cong H_0(X^\cdot)/\Ker f\cong H_0(X^\cdot)/\im \psi\cong\Coker\psi$, we obtain that $K\cong Q$ and $\Dim H_0(Z^\cdot)=\Dim H_0(X^\cdot)+\Dim H_0(Y^\cdot)-2\Dim Q$.

For any $X^\cdot\in C_1(\mathscr{P})$, define $\Dim X^\cdot:=\Dim H_0(X^\cdot)$. For each $\alpha\in \mathbb{N}^n$, let $\H_{\leq \alpha}(C_1(\mathscr{P}))$ be the subspace of $\H(C_1(\mathscr{P}))$ spanned by all $[X^\cdot]$ with $\Dim X^\cdot\leq\alpha$ (in the sense that each component of $\Dim X^\cdot$ is equal or less than the corresponding component of $\alpha$). Then $\H_{\leq \alpha}(C_1(\mathscr{P}))*\H_{\leq\beta}(C_1(\mathscr{P}))\subseteq \H_{\leq \alpha+\beta}(C_1(\mathscr{P}))$, and thus $\H(C_1(\mathscr{P}))$ is an $N^n$-filtered algebra.

By Lemma \ref{kuozhang}, $$\Ext^1_{C_1(\mathscr{P})}(X^\cdot,Y^\cdot)\cong\Hom_A(H_0(X^\cdot),H_0(Y^\cdot))\oplus\Ext^1_{A}(H_0(X^\cdot),H_0(Y^\cdot)).$$
If $f$ in (\ref{czh}) is nonzero, then $\Dim H_0(Z^\cdot)<\Dim H_0(X^\cdot)+\Dim H_0(Y^\cdot)$;
if $f$ in (\ref{czh}) is zero, then $K=Q=0$, thus we obtain the short exact sequence $$\xi: 0\longrightarrow H_0(Y^\cdot)\longrightarrow H_0(Z^\cdot)\longrightarrow H_0(X^\cdot)\longrightarrow0,$$ in this case, $\Dim H_0(Z^\cdot)=\Dim H_0(X^\cdot)+\Dim H_0(Y^\cdot)$; and if $\xi$ is not splitting, by Lemma \ref{end}, $$\dim_k\End_A(H_0(Z^\cdot))<\dim_k\End_A(H_0(X^\cdot)\oplus H_0(Y^\cdot)).$$

Let us give an order on $\mathbb{N}^n\times\mathbb{N}$:
$$(\alpha_1,d_1)\leq(\alpha_2,d_2)\Longleftrightarrow``\alpha_1<\alpha_2"~ \text{or}~``\alpha_1=\alpha_2,d_1\leq d_2".$$

For any $X^{\cdot}\in C_1(\mathscr{P})$, set $\deg X^{\cdot}:=(\Dim H_0(X^\cdot),\dim_k\End_A(H_0(X^\cdot)))$.
Then for any $X^{\cdot},Y^{\cdot}\in C_1(\mathscr{P})$,
$$[X^{\cdot}][Y^{\cdot}]=a_{X^{\cdot}\oplus Y^{\cdot}}[X^{\cdot}\oplus Y^{\cdot}]+\sum\limits_{[Z^{\cdot}]:\deg Z^{\cdot}<\deg X^{\cdot}\oplus Y^{\cdot}}a_{Z^{\cdot}}[Z^{\cdot}].$$

For each fixed total order $\leq$ on the set $S$ such that $[X^\cdot_1]\leq[X^\cdot_2]\leq\cdots\leq[X^\cdot_N]$ are all pairwise non-isomorphic indecomposable objects in $C_1(\mathscr{P})$. Let $X^\cdot=\bigoplus\limits_{i=1}^Na_iX^\cdot_i$, $a_i\in\mathbb{N}$. Then
$$[X^\cdot_1]^{a_1}[X^\cdot_2]^{a_2}\cdots [X^\cdot_N]^{a_N}=a_{X^\cdot}[X^\cdot]+\sum\limits_{[Z^{\cdot}]:\deg Z^{\cdot}<\deg X^{\cdot}}a_{Z^{\cdot}}[Z^{\cdot}].$$
Clearly, $a_{X^{\cdot}}\neq0$. By induction on $\deg X^{\cdot}$, we complete the proof.
\end{proof}

\begin{remark}
Let $Q$ be a Dynkin quiver.
Let $V_0$ be the direct sum of one copy of each indecomposable $A$-module. For each $A$-module $M$, set $d(M):=\dim_k\Hom_A(V_0,M)$. Then for each short exact sequence of $A$-modules $\xi: 0\rightarrow M\rightarrow L\rightarrow N\rightarrow 0$, $d(L)\leq d(M)+d(N)$ and $``="$ holds if and only if $\xi$ is splitting (cf. \cite{FX}).

For any $X^\cdot\in C_1(\mathscr{P})$, define $\Deg X^\cdot=(\Dim H_0(X^\cdot),d(H_0(X^\cdot)))$. Then we can obtain that $\H(C_1(\mathscr{P}))$ is $\mathbb{N}^{n+1}$-filtered with $\H_{\leq \alpha}(C_1(\mathscr{P}))$ being the subspace of $\H(C_1(\mathscr{P}))$ spanned by all $[X^\cdot]$ with $\Deg X^\cdot\leq\alpha$, where $\alpha\in\mathbb{N}^{n+1}$.
\end{remark}

\begin{corollary}
Let $Q$ be a Dynkin quiver. Then the degenerate Hall algebra $\H_1(C_1(\mathscr{P}))$ is isomorphic to the universal enveloping algebra $U(\tilde{\mathfrak{n}})$.
\end{corollary}
\begin{proof}
Since $\H_1(C_1(\mathscr{P}))$ is an associative algebra with Lie subalgebra $\tilde{\mathfrak{n}}$, there is a unique homomorphism $f$ of algebras from the universal enveloping algebra $U(\tilde{\mathfrak{n}})$ to $\H_1(C_1(\mathscr{P}))$ such that the diagram \begin{equation}\label{f}\xymatrix{~~\tilde{\mathfrak{n}}~~\ar@{^{(}->}[r]\ar@{^{(}->}[rd]&U(\tilde{\mathfrak{n}})\ar@{-->}[d]^f\\
&\H_1(C_1(\mathscr{P}))}\end{equation} commutes. By Theorem \ref{PBW} and the PBW-basis theorem of universal enveloping algebra, we conclude that $f$ is an isomorphism.
\end{proof}

\section{Minimal generators of $\H(C_1(\mathscr{P}))$}
From now onwards until the end of the paper, let $Q$ be always a Dynkin quiver.
In this section, we will give a minimal set of generators for the Hall algebra $\H(C_1(\mathscr{P}))$. Then combining with \cite[Theorem 4.3]{RSZ}, we establish a relation between the degenerate Hall algebras $\H_1(A)$ and $\H_1(C_1(\mathscr{P}))$.
\begin{theorem}\label{main result}
$\{[C_{P_i}],[K_{P_i}]~|~1\leq i\leq n\}$ is a minimal set of generators for the Hall algebra $\H(C_1(\mathscr{P}))$.
\end{theorem}
By Theorem \ref{PBW}, we only prove that every $[C_M]$ corresponding to indecomposable non-projective $A$-module $M$ can be generated by the elements in $\{[C_{P_i}],[K_{P_i}]~|~1\leq i\leq n\}$.

Let $M$ be an indecomposable non-projective $A$-module, and fix the minimal projective resolution of $M$: $$0\longrightarrow \Omega\longrightarrow P\longrightarrow M\longrightarrow 0.$$

Before proving Theorem \ref{main result}, we first give the following Lemmas.
\begin{lemma}\label{cp}
Let $P$ be an indecomposable projective $A$-module. Then for any positive integer $m$,
$$[C_{mP}][C_P]=(\frac{q^{m+1}-1}{q-1})[C_{(m+1)P}]+\frac{1}{q^{m-1}}[C_{(m-1)P}][K_P].$$
\end{lemma}
\begin{proof}By Lemma \ref{kuozhang},
$\Ext^1_{C_1(\mathscr{P})}(C_{mP},C_P)\cong\Hom_A(mP,P)\cong k^m$.

Let $f=(f_1,\cdots,f_m)\in\Hom_A(mP,P)$ be nonzero. We assume that $f_i\neq0$ for some $1\leq i\leq m$, then $f_i$ is an isomorphism. Thus, $f$ is a splitting epimorphism, and $\Ker f\cong (m-1)P$. Consider the extension corresponding to $f$:
$0\rightarrow C_P\rightarrow X^{\cdot}\rightarrow C_{mP}\rightarrow0$, then $X^\cdot$ must be isomorphic to $C_{(m-1)P}\oplus K_P$. Hence, $$[C_{mP}][C_P]=F_{C_{mP}C_P}^{C_{(m+1)P}}[C_{(m+1)P}]+F_{C_{mP}C_P}^{C_{(m-1)P}\oplus K_P}[C_{(m-1)P}\oplus K_P].$$
By Riedtmann--Peng formula, it is easy to see that $$F_{C_{mP}C_P}^{C_{(m+1)P}}=\frac{q^{m+1}-1}{q-1}~~~\text{and}~~~ F_{C_{mP}C_P}^{C_{(m-1)P}\oplus K_P}=1.$$
Since $[C_{(m-1)P}][K_P]=F_{C_{(m-1)P}K_P}^{C_{(m-1)P}\oplus K_P}[C_{(m-1)P}\oplus K_P]=q^{m-1}[C_{(m-1)P}\oplus K_P]$, we complete the proof.
\end{proof}

Let $(i_1,i_2,\cdots,i_n)$ be a permutation of $(1,2,\cdots,n)$ such that $\Hom_A(P_{i_s},P_{i_t})\neq 0$ implies that $s\geq t$.
\begin{lemma}\label{cps}
Let $a_i\in\mathbb{N}, 1\leq i\leq n$. Then
for any projective $A$-module $P\cong\oplus_{i=1}^na_iP_i$, we have that
$$[C_P]=\sum\limits_{s_{i_j},t_{i_j},1\leq j\leq n:s_{i_j}+2t_{i_j}=a_{i_j}}a_{t_{i_1
},t_{i_2},\cdots,t_{i_n}}^{s_{i_1},s_{i_2},\cdots,s_{i_n}}\prod_{l=1}^n[C_{P_{i_l}}]^{s_{i_l}}[K_{P_{i_l}}]^{t_{i_l}}$$
for some $a_{t_{i_1},t_{i_2},\cdots,t_{i_n}}^{s_{i_1},s_{i_2},\cdots,s_{i_n}}\in\mathbb{Q}$.
\end{lemma}
\begin{proof}
Using Lemma \ref{cp}, by induction on $m$, we obtain that
$$[C_{mP}]=\sum\limits_{s,t:s+2t=m}a_t^s[C_P]^s[K_P]^t,~~\text{for some}~a_t^s\in\mathbb{Q}.$$
By the fact that $[C_P]=[C_{a_{i_1}P_{i_1}}\oplus C_{a_{i_2}P_{i_2}}\oplus\cdots\oplus C_{a_{i_n}P_{i_n}}]$ and $\Hom_A(a_{i_s}P_{i_s},a_{i_t}P_{i_t})=0$ for any $s<t$, we complete the proof.
\end{proof}
\begin{lemma}\label{bkfj}
Let $0\rightarrow C_P\rightarrow Z^\cdot\rightarrow C_\Omega\rightarrow 0$ be a short exact sequence with $Z^\cdot\in C_1(\mathscr{P})$ indecomposable. Then there exists a monomorphism $f:\Omega\rightarrow P$ such that $\Coker f\cong H_0(Z^\cdot)$. Moreover, $Z^\cdot\cong C_M$.
\begin{proof}
Clearly, $Z^\cdot$ is not projective, since $P$ is not isomorphic to $\Omega$. Hence, we assume that $Z^\cdot=C_L$ for some indecomposable $A$-module $L$. Considering the long exact sequence in homology
$$\xymatrix{P\ar[r]&L\ar[rr]\ar@{->>}[rd]&&\Omega\ar[r]^-{f}&P\ar[rr]\ar@{->>}[rd]&&L\ar[r]&\Omega,\\
&&\Ker f\ar@{^{(}->}[ru]&&&\Coker f\ar@{^{(}->}[ru]&&}$$
we obtain the short exact sequence $0\rightarrow\Coker f\rightarrow L \rightarrow\Ker f\rightarrow 0$, and $L\cong\Ker f\oplus\Coker f$, since $\Ker f\leq\Omega$ is projective. We conclude that $\Ker f=0$ or $\Coker f=0$, since $L$ is indecomposable.

Suppose that $\Coker f=0$, that is, $f$ is an epimorphism. Then we have the short exact sequence $0\rightarrow L \rightarrow \Omega \rightarrow P\rightarrow 0$, and $P\oplus L\cong\Omega$, thus $L$ is projective. By the short exact sequence $0\rightarrow C_P\rightarrow Z^\cdot\rightarrow C_\Omega\rightarrow 0$, we obtain that $L\cong P\oplus\Omega$. This is a contradiction, since $P, \Omega$ are nonzero, and $L$ is indecomposable. Hence, $\Ker f=0$ and $\Coker f\cong L\cong H_0(Z^\cdot)$. Since $\Dim L=\Dim H_0(Z^\cdot)=\Dim P-\Dim\Omega=\Dim M$, and $L, M$ are both indecomposable, we obtain that $L\cong M$, and thus $Z^\cdot=C_L\cong C_M$.

\end{proof}
\end{lemma}

\noindent\textbf{\emph{Proof of Theorem \ref{main result}:}}
By Lemma \ref{kuozhang}, we have that
$$\Ext^1_{C_1(\mathscr{P})}(C_\Omega,C_P)\cong\Hom_A(\Omega,P).$$
For any $f\in\Hom_A(\Omega,P)$, we consider the corresponding extension: \begin{equation}\label{zf}0\longrightarrow C_P\longrightarrow Z^\cdot(f)\longrightarrow C_\Omega\longrightarrow 0.\end{equation}
As before, we know that $\Dim H_0(Z^\cdot(f))=\Dim P+\Dim \Omega-2(\Dim \Omega-\Dim\Ker f)=\Dim M+2\Dim\Ker f\geq \Dim M.$
Hence, if $f$ is not injective, then $\Dim H_0(Z^\cdot(f))>\Dim M$; if $f$ is injective and $Z^\cdot(f)\ncong C_M$, then by Lemma \ref{bkfj}, $Z^\cdot(f)$ is decomposable.

For any $X^\cdot\in C_1(\mathscr{P})$, set $\deg X^\cdot:=(\Dim H_0(X^\cdot),m(X^\cdot))$, where $m(X^\cdot)$ is the number of indecomposable direct summands of $X^\cdot$.

Consider the opposite of the lexicographical order on $\mathbb{N}^n\times\mathbb{N}$: \begin{equation}\label{fx}(\alpha_1,d_1)\leq(\alpha_2,d_2)\Longleftrightarrow``\alpha_1>\alpha_2"~ \text{or}~``\alpha_1=\alpha_2,d_1\geq d_2".\end{equation}
Hence, for any injective $f\in\Hom_A(\Omega,P)$ such that $Z^\cdot(f)\ncong C_M$, or any non-injective $f\in\Hom_A(\Omega,P)$, we have that $\deg Z^\cdot(f)<\deg C_M$.

Let $s$ and $t$ be the number of indecomposable direct summands of $P$ and $\Omega$, respectively. Then for any $0\rightarrow C_P\rightarrow Z^\cdot(f)\rightarrow C_\Omega\rightarrow 0$ as in (\ref{zf}), $Z^\cdot(f)$ has at most $s+t$ indecomposable direct summands.
If $f=0$, then $Z^\cdot(f)\cong C_P\oplus C_\Omega$, in this case, $\Dim H_0(Z^\cdot(f))$ and $m(Z^\cdot(f))$ are both maximal, and thus $\deg Z^\cdot(f)$ is minimal under the order defined in (\ref{fx}). That is, this case is the starting point of our induction.
$$[C_\Omega][C_P]=a[C_P\oplus C_\Omega]+b[C_M]+\sum\limits_{[Z^\cdot]:\deg Z^\cdot<\deg C_M}c_{Z^\cdot}[Z^\cdot].$$
Clearly, $b\neq0$. By induction on $\deg C_M$ together with Lemma \ref{cps}, we complete the proof.

\section{Relations between degenerate Hall algebras $\H_1(A)$ and $\H_1(C_1(\mathscr{P}))$}
In this section, as a first application of Theorem \ref{main result}, we establish a relation between degenerate Hall algebras $\H_1(A)$ and $\H_1(C_1(\mathscr{P}))$.
For any $1\leq i\neq j\leq n$, we
set $$L_{ij}^{\geq2}:=\{\alpha_{ij}~|~\alpha_{ij}~\text{is~a~path~between}~i~
\text{and}~j,~\text{which~is~of~length~at~least~two}\},$$ and define the ideal of $\H_1(C_1(\mathscr{P}))$ $$\mathfrak{I}_0:=\lr{[C_{P_i}][C_{P_j}]-[C_{P_j}][C_{P_i}]~|~L_{ij}^{\geq2}\neq\emptyset,1\leq i\neq j\leq n}.$$
We remark that for any $1\leq i\neq j\leq n$ either $L_{ij}^{\geq2}=\emptyset$ or $|L_{ij}^{\geq2}|=1$, since $Q$ is a Dynkin quiver.
\begin{theorem}\label{gx}
There exists an epimorphism of algebras
$$\xymatrix{\varphi_0: \H_1(A)\otimes k[x_1,x_2,\cdots,x_n]\ar@{->>}[r]&\H_1(C_1(\mathscr{P}))/\mathfrak{I}_0
}$$ defined by $[S_i]\mapsto [C_{P_i}]$, and $x_i\mapsto [K_{P_i}]$.
\end{theorem}
\begin{proof}
It is easy to see that for any $M\in\mod A$ and $P,Q\in\mathscr{P}$ we have that $[K_P][C_M]=[C_M][K_P]$ and $[K_P][K_Q]=[K_Q][K_P]$ in $\H_1(C_1(\mathscr{P}))$.
Then by Theorem \ref{cpg}, we obtain that $\varphi_0$ is a homomorphism of algebras, since it is well-known that $\H_1(A)$ is generated by all $[S_i]$ and the Serre relations (cf. \cite{R90a,R92a}). It follows that $\varphi_0$ is an epimorphism from Theorem \ref{main result}.
\end{proof}

If each vertex of the quiver $Q$ is either a sink or a source, then $Q$ is said to be \emph{bipartite}.
\begin{corollary}
Let $Q$ be a bipartite Dynkin quiver. Then there exists an isomorphism of algebras $\psi_0: \H_1(A)\otimes k[x_1,x_2,\cdots,x_n]\rightarrow \H_1(C_1(\mathscr{P}))$ defined by $[S_i]\mapsto [C_{P_i}]$, and $x_i\mapsto [K_{P_i}]$.
\end{corollary}
\begin{proof}
Since $Q$ is bipartite, we obtain that $L_{ij}^{\geq2}=\emptyset$ for any $1\leq i\neq j\leq n$.
By Theorem \ref{gx}, $\psi_0$ is an epimorphism of algebras.

Let $\mathfrak{n}^+$ be the positive part of the simple Lie algebra associated to $Q$, whose canonical generators are denoted by $e_i,1\leq i\leq n$. Since $Q$ is bipartite, using Theorem \ref{cpg}, we can prove that there exists an isomorphism of Lie algebras $\gamma:\mathfrak{n}^+\rightarrow\tilde{\mathfrak{n}}^+$ defined by $e_i\mapsto [C_{P_i}]$ (cf. \cite[Corollary 4.6]{RSZ}). Hence, it induces an isomorphism of algebras $\tilde{\gamma}:U(\mathfrak{n}^+)\rightarrow U(\tilde{\mathfrak{n}}^+)$ with $e_i\mapsto [C_{P_i}]$. It is well-known that there exists an isomorphism of algebras $\eta:\H_1(A)\rightarrow U(\mathfrak{n}^+)$ defined by $[S_i]\mapsto e_i$ (cf. \cite{R90,R91a}). By the following commutative diagram:
$$\xymatrix@C=1.5in{\H_1(A)\otimes k[x_1,x_2,\cdots,x_n]\ar@{->>}[r]^-{\psi_0}\ar[d]_-{\eta\otimes1}^-{\cong}&\H_1(C_1(\mathscr{P}))\\
U(\mathfrak{n}^+)\otimes k[x_1,x_2,\cdots,x_n]\ar[d]_-{\tilde{\gamma}\otimes1}^-{\cong}&\\U(\tilde{\mathfrak{n}}^+)\otimes k[x_1,x_2,\cdots,x_n]\ar[r]^-{\cong}_-{g}&U(\tilde{\mathfrak{n}})\ar[uu]_-f^-{\cong}}$$where $g$ and $f$ are from $(\ref{fi})$ and $(\ref{f})$, respectively,
we conclude that $\psi_0$ is an isomorphism.
\end{proof}

\section{ Fundamental relations associated to $\H(C_1(\mathscr{P}))$}
Recall that $Q$ is always a Dynkin quiver. In this section, we calculate certain relations in the generators given in Theorem $\ref{main result}$, and obtain some fundamental relations, which have appeared in the Hall algebra of $A$ (cf. \cite{R90a,CD}), in a quotient of $\H(C_1(\mathscr{P}))$.

For any $1\leq i\neq j\leq n$, if $a_{ij}=1$, i.e., there is a path from $i$ to $j$,
then $\Hom_A(P_i,P_j)=0$ and $\Hom_A(P_j,P_i)\cong k$. Note that each nonzero morphism from $P_j$ to $P_i$ is a monomorphism. Take an arbitrary $0\neq f\in\Hom_A(P_j,P_i)$, and set $M_{ij}:=\Coker f$. Then we have a short exact sequence of $A$-modules
\begin{equation}\label{seq}
\xymatrix{0\ar[r]&P_j\ar[r]^-{f}&P_i\ar[r]&M_{ij}\ar[r]&0.}\end{equation} Clearly, $M_{ij}$ is indecomposable, thus $M_{ij}$ is uniquely determined up to isomorphism, which does not depend on the nonzero $f$. Actually, the sequence $(\ref{seq})$ is the minimal projective resolution of $M_{ij}$.

\begin{proposition}\label{jbgx}
For any $1\leq i\neq j\leq n$,

$(1)$~if $a_{ij}=0$, then $[C_{P_i}][C_{P_j}]=[C_{P_j}][C_{P_i}]=[C_{P_i\oplus P_j}]$ in $\H(C_1(\mathscr{P}))$;

$(2)$~if $a_{ij}=1$, then the following relations
\begin{equation}\label{first}
\begin{cases}
[C_{P_j}]^2[C_{P_i}]=(q+1)[C_{P_i\oplus2P_j}]+(q+1)[C_{P_j}\oplus C_{M_{ij}}]+[C_{P_i}\oplus K_{P_j}]\\
[C_{P_j}][C_{P_i}][C_{P_j}]=q(q+1)[C_{P_i\oplus2P_j}]+q[C_{P_j}\oplus C_{M_{ij}}]+q[C_{P_i}\oplus K_{P_j}]\\
[C_{P_i}][C_{P_j}]^2=q^2(q+1)[C_{P_i\oplus2P_j}]+q[C_{P_i}\oplus K_{P_j}]
\end{cases}
\end{equation}
and
\begin{equation}\label{second}
\begin{cases}
[C_{P_i}]^2[C_{P_j}]=q^2(q+1)[C_{P_j\oplus2P_i}]+q[C_{P_j}\oplus K_{P_i}]\\
[C_{P_i}][C_{P_j}][C_{P_i}]=q(q+1)[C_{P_j\oplus2P_i}]+q[C_{P_i}\oplus C_{M_{ij}}]+q[C_{P_j}\oplus K_{P_i}]\\
[C_{P_j}][C_{P_i}]^2=(q+1)[C_{P_j\oplus2P_i}]+(q+1)[C_{P_i}\oplus C_{M_{ij}}]+[C_{P_j}\oplus K_{P_i}]
\end{cases}
\end{equation}
hold in $\H(C_1(\mathscr{P}))$.
\end{proposition}
\begin{proof}
In the whole proof, we will use the Riedtmann--Peng formula together with Proposition 2.4 in \cite{RSZ} to calculate Hall numbers.

$(1)$~Since $a_{ij}=0$, that is, there is no path between $i$ and $j$, we obtain that $$\Hom_A(P_i,P_j)=\Hom_A(P_j,P_i)=0,$$ thus $\Ext^1_{C_1(\mathscr{P})}(C_{P_i},C_{P_j})=\Ext^1_{C_1(\mathscr{P})}(C_{P_j},C_{P_i})=0$. Then it is easy to see that $$F_{C_{P_i}C_{P_j}}^{C_{P_i}\oplus C_{P_j}}=F_{C_{P_j}C_{P_i}}^{C_{P_i}\oplus C_{P_j}}=1.$$

$(2)$~Applying $\Hom_A(P_j,-)$ to the sequence $(\ref{seq})$, we obtain that $\Hom_A(P_j,M_{ij})=0$, thus $\Ext^1_{C_1(\mathscr{P})}(C_{P_j},C_{M_{ij}})\cong\Hom_A(P_j,M_{ij})=0$. Similarly, applying $\Hom_A(-,P_j)$ to the sequence $(\ref{seq})$, we get that $\Ext^1_{C_1(\mathscr{P})}(C_{M_{ij}},C_{P_j})\cong\Ext_A^1(M_{ij},P_j)$ is one-dimensional.

\begin{equation}
\begin{split}
[C_{P_j}]^2&=F_{C_{P_j}C_{P_j}}^{C_{2P_j}}[C_{2{P_j}}]+F_{C_{P_j}C_{P_i}}^{K_{P_j}}[K_{P_j}]\\
&=\frac{1}{q}\cdot\frac{(q^2-q)(q^2-1)}{(q-1)^2}[C_{2{P_j}}]+\frac{q-1}{q}\cdot\frac{(q-1)q}{(q-1)^2}[K_{P_j}]\\
&=(q+1)[C_{2{P_j}}]+[K_{P_j}];
\end{split}
\end{equation}
\begin{equation}
\begin{split}
[C_{P_i}][C_{P_j}]^2&=[C_{P_i}]((q+1)[C_{2{P_j}}]+[K_{P_j}])\\&=(q+1)[C_{P_i}][C_{2{P_j}}]+[C_{P_i}][K_{P_j}]\\
&=(q+1)\cdot\frac{(q-1)(q^2-q)(q^2-1)q^2}{(q-1)(q^2-q)(q^2-1)}[C_{P_i\oplus2{P_j}}]+q[C_{P_i}\oplus K_{P_j}]\\
&=q^2(q+1)[C_{P_i\oplus2{P_j}}]+q[C_{P_i}\oplus K_{P_j}];
\end{split}
\end{equation}
\begin{equation}
\begin{split}
[C_{P_j}][C_{P_i}]&=F_{C_{P_j}C_{P_i}}^{C_{P_i\oplus P_j}}[C_{P_i}\oplus C_{P_j}]+F_{C_{P_j}C_{P_i}}^{C_{M_{ij}}}[C_{M_{ij}}]\\
&=\frac{1}{q}\cdot\frac{(q-1)^2q}{(q-1)^2}[C_{P_i}\oplus C_{P_j}]+\frac{q-1}{q}\cdot\frac{(q-1)q}{(q-1)^2}[C_{M_{ij}}]\\
&=[C_{P_i}\oplus C_{P_j}]+[C_{M_{ij}}];
\end{split}
\end{equation}
\begin{equation}
\begin{split}
[C_{P_j}][C_{P_i}][C_{P_j}]&=([C_{P_i\oplus P_j}]+[C_{M_{ij}}])[C_{P_j}]\\&=[C_{P_i\oplus P_j}][C_{P_j}]+[C_{M_{ij}}][C_{P_j}]\\
&=F_{C_{P_i\oplus P_j}C_{P_j}}^{C_{P_i\oplus2{P_j}}}[C_{P_i\oplus2{P_j}}]+F_{C_{P_i\oplus P_j}C_{P_j}}^{C_{P_i}\oplus K_{P_j}}[C_{P_i}\oplus K_{P_j}]\\&+F_{C_{M_{ij}}C_{P_j}}^{C_{M_{ij}}\oplus C_{P_j}}[C_{P_j}\oplus C_{M_{ij}}]+F_{C_{M_{ij}}C_{P_j}}^{C_{P_i}\oplus K_{P_j}}[C_{P_i}\oplus K_{P_j}]\\
&=\frac{1}{q}\cdot\frac{(q-1)(q^2-q)(q^2-1)q^2}{(q-1)^2q(q-1)}[C_{P_i\oplus2{P_j}}]+\frac{q-1}{q}\cdot\frac{(q-1)^2q^2}{(q-1)^3q}
[C_{P_i}\oplus K_{P_j}]\\&+q[C_{P_j}\oplus C_{M_{ij}}]+\frac{q-1}{q}\frac{(q-1)^2q^2}{(q-1)q(q-1)}[C_{P_i}\oplus K_{P_j}]\\
&=q(q+1)[C_{P_i\oplus2{P_j}}]+q[C_{P_j}\oplus C_{M_{ij}}]+q[C_{P_i}\oplus K_{P_j}];
\end{split}
\end{equation}
\begin{equation}
\begin{split}
[C_{P_j}]^2[C_{P_i}]&=[C_{P_j}][C_{P_i\oplus P_j}]+[C_{P_j}][C_{M_{ij}}]\\
&=F_{C_{P_j}C_{P_i\oplus P_j}}^{C_{P_i\oplus2{P_j}}}[C_{P_i\oplus2{P_j}}]+F_{C_{P_j}C_{P_i\oplus P_j}}^{C_{P_i}\oplus K_{P_j}}[C_{P_i}\oplus K_{P_j}]\\&+F_{C_{P_j}C_{P_i\oplus P_j}}^{C_{P_j}\oplus C_{M_{ij}}}[C_{P_j}\oplus C_{M_{ij}}]+F_{C_{P_j}C_{M_{ij}}}^{C_{P_j}\oplus C_{M_{ij}}}[C_{P_j}\oplus C_{M_{ij}}]\\
&=\frac{1}{q^2}\cdot\frac{(q-1)(q^2-q)(q^2-1)q^2}{(q-1)^3q}[C_{P_i\oplus2{P_j}}]+\frac{(q-1)q}{q^2}\frac{(q-1)^2q^2}{(q-1)^3q}[C_{P_i}\oplus K_{P_j}]\\&+\frac{q-1}{q^2}\cdot\frac{(q-1)q(q-1)q^2}{(q-1)^3q}[C_{P_j}\oplus C_{M_{ij}}]+\frac{1}{q}\frac{(q-1)^2q^3}{(q-1)^2q}[C_{P_j}\oplus C_{M_{ij}}]\\&=(q+1)[C_{P_i\oplus2{P_j}}]+(q+1)[C_{P_j}\oplus C_{M_{ij}}]+[C_{P_i}\oplus K_{P_j}].
\end{split}
\end{equation}
Hence, we complete the proof of the relations in $(\ref{first})$.
\begin{equation}
\begin{split}
[C_{P_i}][C_{P_j}]&=F_{C_{P_i}C_{P_j}}^{C_{P_i\oplus P_j}}[C_{P_i\oplus P_j}]=\frac{(q-1)^2q}{(q-1)^2}[C_{P_i\oplus P_j}]=q[C_{P_i\oplus P_j}];
\end{split}
\end{equation}
\begin{equation}
\begin{split}
[C_{P_i}]^2[C_{P_j}]&=q[C_{P_i}][C_{P_i\oplus P_j}]\\&=q(\frac{1}{q}\cdot\frac{(q-1)(q^2-q)(q^2-1)q^2}{(q-1)^3q}[C_{2P_i\oplus P_j}]+\frac{q-1}{q}\cdot\frac{(q-1)^2q^2}{(q-1)^3q}[C_{P_j}\oplus K_{P_i}])\\&
=q^2(q+1)[C_{2P_i\oplus P_j}]+q[C_{P_j}\oplus K_{P_i}];
\end{split}
\end{equation}
\begin{equation}
\begin{split}
[C_{P_i}][C_{P_j}][C_{P_i}]&=[C_{P_i}][C_{P_i\oplus P_j}]+[C_{P_i}][C_{M_{ij}}]\\
&=q(q+1)[C_{2P_i\oplus P_j}]+[C_{P_j}\oplus K_{P_i}]+q[C_{P_i}\oplus C_{M_{ij}}]+(q-1)[C_{P_j}\oplus K_{P_i}]\\
&=q(q+1)[C_{2P_i\oplus P_j}]+q[C_{P_j}\oplus K_{P_i}]+q[C_{P_i}\oplus C_{M_{ij}}];
\end{split}
\end{equation}
\begin{equation}
\begin{split}
[C_{P_j}][C_{P_i}]^2&=(q+1)[C_{P_j}][C_{2{P_i}}]+[C_{P_j}][K_{P_i}]\\&=
(q+1)(\frac{1}{q^2}\cdot\frac{q^3(q-1)^3(q+1)}{q(q-1)^3(q+1)}[C_{2P_i\oplus P_j}]+\\&\frac{q^2-1}{q^2}\cdot\frac{q^3(q-1)^2}{q(q-1)^3(q+1)}[C_{P_i}\oplus C_{M_{ij}}])+[C_{P_j}\oplus K_{P_i}]\\
&=(q+1)[C_{2P_i\oplus P_j}]+(q+1)[C_{P_i}\oplus C_{M_{ij}}])+[C_{P_j}\oplus K_{P_i}].
\end{split}
\end{equation}
Hence, we complete the proof of the relations in $(\ref{second})$.
\end{proof}
\begin{remark}
$(1)$~As a corollary of Proposition $\ref{jbgx}$, taking $q=1$, we obtain Relation $(a)$ in Theorem $\ref{cpg}$.

$(2)$~There do not exist $x,y,z\in Q(v)$ such that $$x[C_{P_j}]^2[C_{P_i}]+y[C_{P_j}][C_{P_i}][C_{P_j}]+z[C_{P_i}][C_{P_j}]^2=0$$ or
$$x[C_{P_i}]^2[C_{P_j}]+y[C_{P_i}][C_{P_j}][C_{P_i}]+z[C_{P_j}][C_{P_i}]^2=0.$$
That is, we fail to obtain an analogue of quantum Serre relations in $\H(C_1(\mathscr{P}))$.
\end{remark}

Let $\mathfrak{K}_0$ be the ideal of $\H(C_1(\mathscr{P}))$ generated by all $[K_{P_i}]$, $1\leq i\leq n$. Set $$\underline{\H}(C_1(\mathscr{P})):=\H(C_1(\mathscr{P}))/\mathfrak{K}_0.$$ By abuse of notation, in what follows, the image of each $[C_{P_i}]$ under the canonical projection is also denoted by $[C_{P_i}]$.

\begin{proposition}\label{fgx}{\rm\textbf{(Fundamental relations)}}
For any $1\leq i\neq j\leq n$, we have in $\underline{\H}(C_1(\mathscr{P}))$ that

$(1)$~if $a_{ij}=0$, then $[C_{P_i}][C_{P_j}]=[C_{P_j}][C_{P_i}]$;

$(2)$~if $a_{ij}=1$, then
$$[C_{P_i}][C_{P_j}]^2-(q+1)[C_{P_j}][C_{P_i}][C_{P_j}]+q[C_{P_j}]^2[C_{P_i}]=0$$
and
$$q[C_{P_j}][C_{P_i}]^2-(q+1)[C_{P_i}][C_{P_j}][C_{P_i}]+[C_{P_i}]^2[C_{P_j}]=0.$$
\end{proposition}
\begin{proof}
$(1)$~This is clear.

$(2)$~By Proposition $\ref{jbgx}$, noting that $[C_{P_i}\oplus K_{P_j}]=[K_{P_j}][C_{P_i}]\in\mathfrak{K}_0$ and $[C_{P_j}\oplus K_{P_i}]=[C_{P_j}][K_{P_i}]\in\mathfrak{K}_0$, we obtain in $\underline{\H}(C_1(\mathscr{P}))$ that
\begin{equation}
\begin{cases}
[C_{P_j}]^2[C_{P_i}]=(q+1)[C_{P_i\oplus2P_j}]+(q+1)[C_{P_j}\oplus C_{M_{ij}}]\\
[C_{P_j}][C_{P_i}][C_{P_j}]=q(q+1)[C_{P_i\oplus2P_j}]+q[C_{P_j}\oplus C_{M_{ij}}]\\
[C_{P_i}][C_{P_j}]^2=q^2(q+1)[C_{P_i\oplus2P_j}]
\end{cases}
\end{equation}
and
\begin{equation}
\begin{cases}
[C_{P_i}]^2[C_{P_j}]=q^2(q+1)[C_{P_j\oplus2P_i}]\\
[C_{P_i}][C_{P_j}][C_{P_i}]=q(q+1)[C_{P_j\oplus2P_i}]+q[C_{P_i}\oplus C_{M_{ij}}]\\
[C_{P_j}][C_{P_i}]^2=(q+1)[C_{P_j\oplus2P_i}]+(q+1)[C_{P_i}\oplus C_{M_{ij}}].
\end{cases}
\end{equation}
Then it is easy to obtain the desired relations.
\end{proof}

\section{Quantum groups and Hall algebras}
In this section, in order to acquire the quantum Serre relations via $\H(C_1(\mathscr{P}))$, we need to consider the twisted versions of Hall algebras.

For any $M, N\in \mod A$, define $$\lr{M,N}:=\dim_k{\Hom_{A}(M,N)}-\dim_k{\Ext_{A}^{1}(M,N)}.$$
It induces a bilinear form
$$\lr{\cdot ,\cdot }: K(A)\times K(A)\longrightarrow \mathbb{Z},$$ known as the \emph{Euler form}. The \emph{Ringel--Hall algebra} $\H_{\tw}(A)$ of $A$ is the same vector space as $\H(A)$ but with twisted multiplication defined by $$[M]\ast[N]=v^{\lr{{M},\,{N}}}\sum\limits_{[L]}F_{MN}^L[L].$$

Let $U_v(\mathfrak{n}^+)$ be the positive part of the quantum group associated to the Dynkin quiver $Q$, whose canonical generators are denoted by $E_i, 1\leq i\leq n$. The following theorem is well-known.
\begin{theorem}\label{lzq}
There exists an isomorphism of algebras $$R:U_v(\mathfrak{n}^+)\longrightarrow\H_{\tw}(A)$$ defined on generators by $R(E_i)=[S_i]$.
\end{theorem}

Let us give a twisted version of the Hall algebra $\H(C_1(\mathscr{P}))$.
\begin{definition}
The \emph{twisted Hall algebra} $\H_{\tw}(C_1(\mathscr{P}))$ of $C_1(\mathscr{P})$ is the same vector space as $\H(C_1(\mathscr{P}))$ but with twisted multiplication defined by $$[X^\cdot]\ast[Y^\cdot]=v^{-\lr{{Y_0},\,{X_0}}}\sum\limits_{[Z^\cdot]}F_{X^\cdot Y^\cdot}^{Z^\cdot}[Z^\cdot],$$
where $X^\cdot =(X_0,c),Y^\cdot=(Y_0,d)\in C_1(\mathscr{P})$.
\end{definition}
It is easy to see that $\H_{\tw}(C_1(\mathscr{P}))$ is still an associative algebra.
\begin{remark}
Let $\mathfrak{H}$ be the subalgebra of $\H_{\tw}(C_1(\mathscr{P}))$ generated by all $[K_{P_i}]$, $1\leq i\leq n$. Then $\mathfrak{H}\cong k[x_1,x_2,\cdots,x_n]$, and moreover it is contained in the center of $\H_{\tw}(C_1(\mathscr{P}))$. That is, for any $P\in\mathscr{P}$ and $X^\cdot\in C_1(\mathscr{P})$, we have that $[K_P]\ast[X^\cdot]=[X^\cdot]\ast[K_P]$ (This can be easily verified by Riedtmann--Peng formula together with Proposition 2.4 in \cite{RSZ}).
\end{remark}
Let $\mathfrak{K}_1$ be the ideal of $\H_{\tw}(C_1(\mathscr{P}))$ generated by all $[K_{P_i}]$, $1\leq i\leq n$. Set $$\underline{\H}_{\tw}(C_1(\mathscr{P})):=\H_{\tw}(C_1(\mathscr{P}))/\mathfrak{K}_1.$$
Let us reformulate Propositions $\ref{jbgx}$ and $\ref{fgx}$.

\begin{proposition}\label{tjbgx}
For any $1\leq i\neq j\leq n$,

$(1)$~if $a_{ij}=0$, then $[C_{P_i}]\ast[C_{P_j}]=[C_{P_j}]\ast[C_{P_i}]=[C_{P_i\oplus P_j}]$ in $\H_{\tw}(C_1(\mathscr{P}))$;

$(2)$~if $a_{ij}=1$, then the following relations
\begin{equation}\label{gx1}
\begin{cases}
[C_{P_j}]^{*2}*[C_{P_i}]=(v+v^{-1})[C_{P_i\oplus2P_j}]+(v+v^{-1})[C_{P_j}\oplus C_{M_{ij}}]+v^{-1}[C_{P_i}\oplus K_{P_j}]\\
[C_{P_j}]*[C_{P_i}]*[C_{P_j}]=(q+1)[C_{P_i\oplus2P_j}]+[C_{P_j}\oplus C_{M_{ij}}]+[C_{P_i}\oplus K_{P_j}]\\
[C_{P_i}]\ast[C_{P_j}]^{\ast2}=v(q+1)[C_{P_i\oplus2P_j}]+v^{-1}[C_{P_i}\oplus K_{P_j}]
\end{cases}
\end{equation}
and
\begin{equation}\label{gx2}
\begin{cases}
[C_{P_i}]^{\ast2}*[C_{P_j}]=v(q+1)[C_{P_j\oplus2P_i}]+v^{-1}[C_{P_j}\oplus K_{P_i}]\\
[C_{P_i}]\ast[C_{P_j}]\ast[C_{P_i}]=(q+1)[C_{P_j\oplus2P_i}]+[C_{P_i}\oplus C_{M_{ij}}]+[C_{P_j}\oplus K_{P_i}]\\
[C_{P_j}]\ast[C_{P_i}]^{\ast2}=(v+v^{-1})[C_{P_j\oplus2P_i}]+(v+v^{-1})[C_{P_i}\oplus C_{M_{ij}}]+v^{-1}[C_{P_j}\oplus K_{P_i}]
\end{cases}
\end{equation}
hold in $\H_{\tw}(C_1(\mathscr{P}))$.
\end{proposition}
\begin{proof}
(1)~$[C_{P_i}]\ast[C_{P_j}]=v^{-\lr{P_j,P_i}}[C_{P_i}][C_{P_j}]=[C_{P_i}][C_{P_j}]$. Similarly, $[C_{P_j}]\ast[C_{P_i}]=[C_{P_j}][C_{P_i}]$;

(2)~$[C_{P_j}]^{\ast2}*[C_{P_i}]=v^{-(\lr{P_i,P_j}+\lr{\hat{P}_i+\hat{P}_j,\hat{P}_j})}[C_{P_j}]^{2}[C_{P_i}]=v^{-1}[C_{P_j}]^{2}[C_{P_i}]$. Similarly, we can prove the other.
\end{proof}

\begin{proposition}\label{qsr}{\rm\textbf{(Quantum Serre relations)}}
For any $1\leq i\neq j\leq n$, we have in $\underline{\H}_{\tw}(C_1(\mathscr{P}))$ that

$(1)$~if $a_{ij}=0$, then $[C_{P_i}]\ast[C_{P_j}]=[C_{P_j}]\ast[C_{P_i}]$;

$(2)$~if $|a_{ij}|=1$, then
\begin{equation}\label{Serre1}[C_{P_i}]\ast[C_{P_j}]^{\ast2}-(v+v^{-1})[C_{P_j}]\ast[C_{P_i}]\ast[C_{P_j}]+[C_{P_j}]^{\ast2}\ast[C_{P_i}]=0\end{equation}
and
\begin{equation}\label{Serre2}[C_{P_j}]\ast[C_{P_i}]^{\ast2}-(v+v^{-1})[C_{P_i}]\ast[C_{P_j}]\ast[C_{P_i}]+[C_{P_i}]^{\ast2}\ast[C_{P_j}]=0.\end{equation}
\end{proposition}
\begin{proof}
$(2)$~We only need to note that $(\ref{Serre1})$ and $(\ref{Serre2})$ are symmetric with respect to $i$ and $j$, then by Proposition \ref{tjbgx}(2), we complete the proof.
\end{proof}

Define the ideal of $\underline{\H}_{\tw}(C_1(\mathscr{P}))$ $$\mathfrak{I}_1:=\lr{[C_{P_i}]\ast[C_{P_j}]-[C_{P_j}]\ast[C_{P_i}]~|~L_{ij}^{\geq2}\neq\emptyset,1\leq i\neq j\leq n}.$$
\begin{theorem}
There exists an epimorphism of algebras
$$\xymatrix{\varphi: \H_{\tw}(A)\ar@{->>}[r]&\underline{\H}_{\tw}(C_1(\mathscr{P}))/\mathfrak{I}_1
}$$ defined by $[S_i]\mapsto [C_{P_i}]$.
\end{theorem}
\begin{proof}
We only need to note that $\H_{\tw}(A)$ is generated by all $[S_i]$ and the quantum Serre relations. Then combining Proposition $\ref{qsr}$ with Theorem $\ref{main result}$, we complete the proof.
\end{proof}

\begin{corollary}\label{hgx}
Let $Q$ be a bipartite Dynkin quiver. Then there exists an isomorphism of algebras $\psi: \H_{\tw}(A)\rightarrow \underline{\H}_{\tw}(C_1(\mathscr{P}))$ defined by $[S_i]\mapsto [C_{P_i}]$.
\end{corollary}
\begin{proof}
Consider
$$\psi':\underline{\H}_{\tw}(C_1(\mathscr{P}))\rightarrow\H_{\tw}(A), [C_M]\mapsto[\top M],$$ then it is easy to see that $\psi'\psi=1$, thus $\psi$ is injective.
\end{proof}

Combining Corollary \ref{hgx} with Theorem \ref{lzq}, we obtain the following
\begin{corollary}
Let $Q$ be a bipartite Dynkin quiver. Then there exists an isomorphism of algebras $\tilde{\psi}: U_v(\mathfrak{n}^+)\rightarrow \underline{\H}_{\tw}(C_1(\mathscr{P}))$ defined by $E_i\mapsto [C_{P_i}]$.
\end{corollary}
\section*{Acknowledgments}

The author is grateful to Qinghua Chen, Shiquan Ruan and Jie Sheng for their stimulating discussions and valuable comments. He also would like to thank Professor Bangming Deng for his encouragement and help. After publishing the paper on arXiv, the author is pleasantly informed that Ming Lu and Weiqiang Wang have also some similar results under different framework in their forthcoming paper.

\end{document}